\documentclass[12pt,A4,reqno]{amsart}
\usepackage{amsfonts}
\usepackage{mathrsfs}

\usepackage{amssymb}
\usepackage{amsmath,amscd}
\usepackage{color}
\usepackage{epsf}
\usepackage{graphicx}
\usepackage{xypic}

\theoremstyle{plain}
\newtheorem{thm}{Theorem}[section]

\newtheorem{prop}[thm]{Proposition}

\theoremstyle{definition}
\newtheorem{defi}[thm]{Definition}
\newtheorem{rem}[thm]{Remark}

\DeclareMathOperator{\link}{link}

\newcommand{\R}{\mathbb R}
\newcommand{\Z}{\mathbb Z}

\newcommand{\nn}{\vskip 0.2cm}
\newcommand{\n}{\vskip 0.1cm}

\begin{document}

\title [\ ] {On descriptions of products of simplices}

\author[L.~Yu]{Li Yu}
\address{Department of Mathematics and IMS, Nanjing University, Nanjing, 210093, P.R.China}
 \email{yuli@nju.edu.cn}

\author[M.~Masuda]{Mikiya Masuda}
\address{Department of Mathematics, Osaka City University, Sugimoto, Sumiyoshi-ku, 
        Osaka 558-8585, Japan}
\email{masuda@sci.osaka-cu.ac.jp}

\thanks{2010 \textit{Mathematics Subject Classification}.
  52B11, 52B70, 53C21, 53C23, 57S17.\\
  This work is partially supported by the
 NSFC (Grant No.11371188) and
 the PAPD (priority academic program development) of Jiangsu higher education institutions.}


\keywords{Convex polytope, product of simplices, moment-angle complex}

 \begin{abstract}
   We give several new criteria to judge whether a simple convex polytope in a Euclidean space
    is combinatorially 
   equivalent to a product of simplices. These criteria are mixtures of
    combinatorial, geometrical and topological conditions that are inspired by
    the ideas from toric topology.
  \end{abstract}

\maketitle

  \section{Background}
      A convex polytope $P$ 
       is the convex hull of a finite set of points
      in a Euclidean space $\R^d$. The dimension of $P$ is the dimension of the
      affine hull of these points. 
      Any codimension-one face of $P$ is called a \emph{facet} of $P$.
      We call an $n$-dimensional convex polytope $P$ \emph{simple} if each vertex of
      $P$ is the intersection of exactly $n$ different facets of $P$.
       Two convex polytopes are \emph{combinatorially equivalent}
       if their face lattices are isomorphic.
       Topologically, combinatorial equivalence
       corresponds to the existence of a (piecewise linear)
        homeomorphism between the two polytopes that
     restricts to homeomorphisms between their facets, and hence all their faces 
          (see~\cite[Chapter 2.2]{Ziegler95}).\nn

       In this paper, we will give several new criteria to judge whether a simple 
     convex polytope is combinatorially 
     equivalent to product of simplices 
    (Theorem~\ref{Thm:New-Comb} and Theorem~\ref{thm:Polytope-product-1}). 
     Some of these criteria are purely combinatorial, while
     others are phrased in geometrical or topological terms. 
     These criteria are mainly inspired by the ideas from  
     toric topology. So in the following we first explain
     some basic constructions and facts 
     in toric topology that are relevant to our discussion.\nn

  An \emph{abstract simplicial complex} on a set $[m]=\{ v_1,\cdots, v_m\}$ 
  is a collection $K$ of subsets $\sigma\subseteq [m]$ such that
  if $\sigma\in K$, then any subset of $\sigma$ also belongs to $K$.
  We always assume that the empty set belongs to $K$ and refer
  to $\sigma\in K$ as a \emph{simplex} of $K$.
  In particular, one-element simplices are called \emph{vertices} of of $K$. 
  If $K$ contains all one-element subsets of $[m]$, then we say that $K$
  is a simplicial complex \emph{on the vertex set} $[m]$.
  To avoid ambiguity in our argument, 
  we also use $V(K)$ and $V(\sigma)$ to refer to the vertex sets of $K$ and 
  any simplex $\sigma$ in $K$. 
 \n
   
    Any abstract simplicial complex $K$ admits a
  \emph{geometric realization} in some Euclidean space. 
  But sometimes we also use $K$ to denote its geometric realization 
  when the meaning is clear in the context.\n

  Given a finite abstract simplicial complex $K$ on a set $[m]$ and 
  a pair of spaces $(X,A)$ with $A\subset X$, we can construct a topological space 
  $(X,A)^{K}$ by:
    \begin{equation}\label{Equ:Construction}
     (X,A)^{K}= \bigcup_{\sigma\in K} (X,A)^{\sigma}, \ \text{where}\
      (X,A)^{\sigma}= \prod_{v_j\in \sigma} X \times \prod_{v_j\notin \sigma} A. 
    \end{equation}  
  Here $\prod$ means Cartesian product.  
  So $(X,A)^{K}$ is a subspace of the Cartesian product of
  $m$ copies of $X$. It is called the \emph{polyhedral product}
   or the \emph{generalized moment-angle complex}
   of $K$ and $(X,A)$. 
  In particular, $\mathcal{Z}_{K}
    =(D^2,S^1)^{K}$ and $\R\mathcal{Z}_{K}
    =(D^1,S^0)^{K}$ are called the 
  \emph{moment-angle complex} and the \emph{real moment-angle complex} of $K$,
   respectively (see~\cite[Section 4.1]{BP15}). The natural  
   actions of $(\Z_2)^m$ on $(D^1)^m$ and $(S^1)^m$ on $(D^2)^m$
    induce \emph{canonical actions} of $(\Z_2)^m$ on $\R\mathcal{Z}_{K}$
   and $(S^1)^m$ on $\mathcal{Z}_{K}$, respectively. \n
   
  When $K$ is the boundary of the dual of a simple convex polytope $P$, 
  the $\mathcal{Z}_{K}$ and $\R\mathcal{Z}_{K}$ are closed manifolds, also denoted by
  $\mathcal{Z}_P$ and $\R\mathcal{Z}_{P}$ respectively. In this case, 
  $\mathcal{Z}_P$ and $\R\mathcal{Z}_{P}$ are called the \emph{moment-angle manifold} and 
  the \emph{real moment-angle manifold} of $P$, respectively (see~\cite[Section 6.1]{BP02}). 
  These manifolds
    can be constructed in another way as described below 
  (see~\cite[Construction 4.1]{DaJan91}).\n
  
   Let $P^n$ be an $n$-dimensional simple convex polytope.
   Let $\mathcal{F}(P^n) =\{F_1,\cdots, F_m \}$ 
   be the set of facets of $P^n$.
  Let $\{e_1,\cdots, e_m\}$ be a basis of $(\Z_2)^m$ and define a map 
   $\lambda : \mathcal{F}(P^n) \rightarrow
    (\Z_2)^m$ by $\lambda(F_i)=e_i$. Then we can construct a space
      \begin{equation} \label{Equ:glue-back}
       M(P^n,\lambda) := P^n\times (\Z_2)^m \slash \sim
     \end{equation}
      where $(p,g) \sim (p',g')$ if and only if $p=p'$ and $g^{-1}g' \in G^{\lambda}_p$
   where $G^{\lambda}_p$ is the subgroup of $(\Z_2)^m$ generated by
  the set $\{ \lambda(F_i) \, |\, p\in F_i \}$. Let $\pi_{\lambda}: 
   M(P^n,\lambda) \rightarrow P^n$ be the quotient map. 
  One can show that $\R \mathcal{Z}_{P^n}$ is homeomorphic to $M(P^n,\lambda)$ 
  and the canonical action of $(\Z_2)^m$ on $\R \mathcal{Z}_{P^n}$ can be written
  on $M(P^n,\lambda)$ as:
   \begin{equation} \label{Equ:Canon-Action}
    g' \cdot [(p,g)] = [(p,g'+g)], \ p\in P^n,\ g,g'\in (\Z_2)^m. 
   \end{equation} 
     
    The moment-angle manifold $\mathcal{Z}_{P^n}$ can be
    similarly constructed from $P^n$ and a map $\Lambda: \mathcal{F}(P^n) \rightarrow
    \Z^m$ where $\{\Lambda(F_1),\cdots,\Lambda(F_m)\}$ is a unimodular basis of $\Z^m$.
    Indeed, if we indentify the torus $(S^1)^m =\R^m\slash \Z^m$, then we have
    \begin{equation} \label{Equ:glue-back-complex}
       \mathcal{Z}_{P^n} \cong P^n\times (S^1)^m \slash \sim
     \end{equation} 
     where $(p,g) \sim (p',g')$ if and only if $p=p'$ and $g^{-1}g' \in T^{\lambda}_p$
   where $T^{\lambda}_p$ is the subtorus of $(S^1)^m$ determined by
   the linear subspace of $\R^m$ spanned by the set $\{ \Lambda(F_i) \, |\, p\in F_i \}$.
   
    In addition, $\R \mathcal{Z}_{P^n}$ and 
    $\mathcal{Z}_{P^n}$ are smooth manifolds.
     In fact, there exists an equivariant smooth structure on 
     $\R \mathcal{Z}_{P^n}$ 
     (or $\mathcal{Z}_{P^n}$) with respect to the canonical $(\Z_2)^m$-action 
     (or $(S^1)^m$-action).
      The reader is referred 
     to~\cite[Ch.6]{BP02} or~\cite[Ch.6]{BP15} for
     the discussion of smooth structures on (real) moment-angle manifolds.
     Moreover, for any proper face $f$ of $P^n$, 
     $\pi^{-1}_{\lambda}(f)$ is an embedded
      closed smooth submanifold of $\R \mathcal{Z}_{P^n}$ 
      which is the fixed point set of the subgroup of $(\Z_2)^m$ generated by
     $\{ \lambda(F_i) \, |\, f\in F_i \}$ under the canonical $(\Z_2)^m$-action. 
        
    \vskip .6cm
  
    \section{Descriptions of products of simplices}
    
     For any $k\in\mathbb{N}$, let 
   $\Delta^k$ denote the standard $k$-dimensional simplex, which is
    $$\Delta^k = \{ (x_1,\cdots, x_k, x_{k+1})\in \R^{k+1}\,|\, x_1+\cdots +x_{k+1} = 1,
     x_1,\cdots, x_{k+1}\geq 0  \}.$$
      For any $n_1,\cdots, n_q\in \mathbb{N}$, consider 
       $\Delta^{n_1}\times\cdots\times\Delta^{n_q}$ as a product of 
       $\Delta^{n_1},\cdots, \Delta^{n_q}$
      in the Cartesian product $\R^{n_1+1}\times\cdots\times \R^{n_q+1}$.\n   
    
    Next, we first list some descriptions of
    products of simplices that appeared in Wiemeler's paper~\cite{Wiem15}.\n

      \begin{thm}[Wiemeler~\cite{Wiem15}] \label{thm:Wiem}
      Let $P^n$ be an $n$-dimensional simple convex polytope with $m$ facets, $n\geq 3$.     
      Then the following statements are equivalent.
    \begin{itemize}
       
      \item[(\textbf{a})] $P^n$ is combinatorially equivalent to a product of simplices.\n

      \item[(\textbf{b})] Any $2$-dimensional face of $P^n$ is either a $3$-gon or a
               $4$-gon.\n
                             
      \item[(\textbf{c})] There exists a quasitoric manifold $M^{2n}$ over $P^n$ which admits 
         a nonnegatively curved Riemannian metric 
         that is invariant under the canonical $(S^1)^n$-action 
           on $M^{2n}$.
   \end{itemize}        
    \end{thm}
        
        A \emph{quasitoric manifold} $M^{2n}$ over $P^n$ is the
        quotient space of $\mathcal{Z}_{P^n}$ under a free action of 
        a rank $m-n$ toral subgroup of $(S^1)^m$ (see~\cite{DaJan91}). 
        There is a canonical $(S^1)^n$-action on $M^{2n}$ induced from the
        canonical action of $(S^1)^m$ on $\mathcal{Z}_{P^n}$, which makes $M^n$ a 
        \emph{torus manifold} (see~\cite{HattMas03}). \n
   
        Theorem~\ref{thm:Wiem}(\textbf{b}) is a 
         corollary of~\cite[Proposition 4.5]{Wiem15} and
          Theorem~\ref{thm:Wiem}(\textbf{c}) is a corollary of~\cite[Lemma 4.2]{Wiem15}.
  Note that Theorem~\ref{thm:Wiem}(\textbf{b})
    is a particularly useful description of products of simplices.
   Indeed, the proofs of many other descriptions of products of simplices in this paper
   boil down to this one first. But the proof of~\cite[Proposition 4.5]{Wiem15}
    is a little long and not particularly easy to follow. We will give a shorter proof 
    of Theorem~\ref{thm:Wiem}(\textbf{b}) in the appendix to make 
    our paper more self-contained. \nn 
    
         Next, we give more descriptions of products of simplices from 
    combinatorial and topological viewpoints. For convenience let
     us introduce some notations first.
    
    \begin{itemize}
    \item For any topological space $X$ and any field $\mathbf{k}$, let
    $$\mathrm{hrk}(X;\mathbf{k}) = \sum^{\infty}_{i=0}\dim_{\mathbf{k}} H^i(X;\mathbf{k}).$$
    
    \item For any vertex $v$ in a simplicial complex $K$,
    we denote by $\mathrm{link}_Kv$ the \emph{link} of $v$ in $K$.  
    We denote a simplex spanned by vertices $v_0,v_1,\dots,v_p$ in $K$ by 
       $[v_0,v_1,\dots, v_p]$ and its boundary complex by $\partial[v_0,v_1,\dots, v_p]$.  
     \end{itemize}     
    \n
    
     In addition, for a simplicial complex $K$ on the vertex set $[m]=\{v_1,\cdots, v_m\}$, 
    we can define a new simplicial complex $L(K)$ from $K$, called the
   \emph{double of $K$}, where $L(K)$ is a simplicial complex on the vertex set
   $[2m]=\{ v_1,v'_1, \cdots, v_m,v'_m \}$
    determined by the following condition:
  $\omega \subset [2m]$ is a minimal (by inclusion) missing simplex of $L(K)$ if and only if 
  $\omega$ is of the form $\{v_{i_1}, v'_{i_1}, \cdots, v_{i_k}, v'_{i_k}\}$ where
  $\{ v_{i_1},\cdots, v_{i_k}\}$ is a minimal 
  missing simplex of $K$. Note that any minimal missing
  simplex in $L(K)$ must have even number of vertices.
  The double of $K$ is a special case of iterated simplicial wedge construction (also called
   simplicial wedge $J$-construction). Indeed, 
   by the notation introduced in~\cite{BBCG-10}, $L(K)=K(2,\cdots, 2)$.\n
  
   The following are some basic facts about $L(K)$ (see Ustinovsky~\cite{Uto09,Uto11}). \n
   \begin{itemize}
     \item $\dim(L(K)) = m + \dim(K)$ (\cite[Lemma 1.2]{Uto11}).\n
    
     \item $L(K_1 * K_2) = L(K_1) * L(K_2)$ (here $*$ is the join of two simplical complexes).\n
    
    \item If $K =\partial P^*$ where $P^*$ is the simplicial polytope dual to
            a simple convex polytope $P$,
        then $L(K) =\partial L(P)^*$ where $L(P)$ is a simple convex polytope called
        the \emph{double of $P$} (see~\cite{Uto09} for the construction of $L(P)$).\n
        
    \item $L(\partial \Delta^{k})=\partial \Delta^{2k+1}$.
   \end{itemize}
 \nn
 
  The following are some easy or well known facts on products of simplices.
   We want to list them and give a simple proof
  for reference.
    
     \begin{prop} \label{Prop:Polytope-product-2}
      Let $P$ be an $n$-dimensional simple polytope with $m$ facets and 
      let
       $K$ be the boundary of the simplicial polytope dual to $P$. Then the
        following statements are all equivalent.
    \begin{itemize}
       
     \item[(\textbf{a})] $P$ is combinatorially equivalent to a product of simplices.\n
         
     \item[(\textbf{b})] $K$ is simplicially isomorphic to 
               $\partial \Delta^{n_1} * \cdots * \partial \Delta^{n_q}$ for some 
               $n_1\cdots, n_q \in \mathbb{N}$.\n
               
     \item[(\textbf{c})] The vertex sets of all the minimal missing faces of $K$ form
      a partition of $V(K)$.\n          
      
     \item[(\textbf{d})] $L(K)$ is simplicially isomorphic to
          $\partial \Delta^{l_1} * \cdots * \partial \Delta^{l_q}$ for some 
          $l_1\cdots, l_q\in\mathbb{N}$.\n

     \item[(\textbf{e})] There exists some field $\mathbf{k}$ so that
                      $\mathrm{hrk}(\R\mathcal{Z}_K;\mathbf{k})=2^{m-\dim(K)-1}$, 
                  or equivalently $\mathrm{hrk}(\R\mathcal{Z}_P;\mathbf{k})=2^{m-n}$.\n

    \item[(\textbf{f})] There exists some field $\mathbf{k}$ so that
                      $\mathrm{hrk}(\mathcal{Z}_K;\mathbf{k})=2^{m-\dim(K)-1}$, 
                   or equivalently
                    $\mathrm{hrk}(\mathcal{Z}_P;\mathbf{k})=2^{m-n}$.
  \end{itemize}
 \end{prop}

    \begin{proof}
       The equivalences of $(\textbf{a}) \Leftrightarrow (\textbf{b})$ and 
       $(\textbf{b}) \Leftrightarrow (\textbf{c})$ 
        are easy to see.
       \nn
     
   $(\textbf{b}) \Rightarrow (\textbf{d}).$ \
     If $K=\partial \Delta^{n_1} * \cdots * \partial \Delta^{n_q}$, then 
     $$L(K) = L(\partial \Delta^{n_1} * \cdots * \partial \Delta^{n_q})=
     L(\partial \Delta^{n_1}) * \cdots * L(\partial \Delta^{n_q}) = 
         \partial \Delta^{2n_1+1} * \cdots * \partial \Delta^{2n_q+1}. $$
    
    $(\textbf{d}) \Rightarrow (\textbf{c}).$ \ 
    Suppose $L(K)=\partial \Delta^{l_1} * \cdots * \partial \Delta^{l_q}$.
       Notice that for each $1\leq j \leq q$, 
       $\Delta^{l_j}$ is a minimal missing simplex of $L(K)$. So $\Delta^{l_j}$ 
       must have even number of vertices, which implies that $l_j$ is an odd integer.
       Then by $(\textbf{b}) \Leftrightarrow (\textbf{c})$,
        the vertex sets of all the minimal missing faces
       of $L(K)$ form a partition of $V(L(K))$. This forces the vertex sets of all the
       minimal missing faces of $K$ to form a partition of $V(K)$ as well, which is (c).
          \nn

  $(\textbf{a}) \Rightarrow (\textbf{e})$ and $(\textbf{f}).$
    If $P=\Delta^{n_1}\times\cdots\times \Delta^{n_q}$, $n_1+\cdots + n_q =n$,
    then $$\mathcal{Z}_P = S^{2n_1+1}\times \cdots \times S^{2n_q+1},\ \ 
    \R\mathcal{Z}_P=S^{n_1}\times \cdots \times S^{n_q}.$$
    The number of facets of $P$ is $m=n+q$. It is clear that for any field $\mathbf{k}$,
    $$\mathrm{hrk}(\mathcal{Z}_P;\mathbf{k}) = 
       \mathrm{hrk}(\R\mathcal{Z}_P;\mathbf{k})=
       2^q=2^{m-n}.$$

 $(\textbf{e}) \Rightarrow (\textbf{a}).$
   For any vertex $v$ of $K$, let $m_v$ be the number of vertices in
   $\mathrm{link}_K v$. According to the 
    proof of~\cite[Theorem 3.2]{Uto11} (note that 
     the argument there works for any coefficient), there is a subspace $X$ of 
   $\R\mathcal{Z}_K$ so that 
   $$\mathrm{hrk}(\R\mathcal{Z}_P;\mathbf{k})  = \mathrm{hrk}(\R\mathcal{Z}_K;\mathbf{k}) \geq \mathrm{hrk}(X;\mathbf{k}), $$
   where $X$ is the disjoint union of $2^{m-m_v-1}$ copies of 
   $\R\mathcal{Z}_{\mathrm{link}_K v}$.
   So we have
   $$2^{m-n} = \mathrm{hrk}(\R\mathcal{Z}_K;\mathbf{k}) \geq 2^{m-m_v-1}
   \mathrm{hrk}(\R\mathcal{Z}_{\mathrm{link}_K v};\mathbf{k}). $$
   Then $\mathrm{hrk}(\R\mathcal{Z}_{\mathrm{link}_K v};\mathbf{k}) \leq 2^{m_v-n+1}$.
   On the other hand, ~\cite[Theorem 3.2]{Uto11} tells us that
   $\mathrm{hrk}(\R\mathcal{Z}_{\mathrm{link}_K v};\mathbf{k}) \geq 2^{m_v-n+1}$
   (since $\dim(\mathrm{link}_K v)=n-2$). So we obtain
   $$\mathrm{hrk}(\R\mathcal{Z}_{\mathrm{link}_K v};\mathbf{k}) = 2^{m_v-n+1}.$$
   
   Note if $v$ is the vertex corresponding to a facet $F$ of $P$, then 
      $\R\mathcal{Z}_{\mathrm{link}_K v} =\R\mathcal{Z}_F$. So we have shown that
     if the condition $(\textbf{e})$ holds for $P$, it should hold
      for any facet of $P$ as well.
    \n
    
    By iterating the above argument, we deduce that
     the condition $(\textbf{e})$ holds for all the two dimensional faces of $P$.
     It is not hard to see that the real moment-angle manifold of a $k$-gon is
     a closed connected orientable surface with genus $1+(k-4)2^{k-3}$ 
     (see~\cite[Proposition 4.1.8]{BP15}).
    So any $2$-dimensional face of $P$ is either
  a $3$-gon or a $4$-gon. Then by Theorem~\ref{thm:Wiem}(b),
   the polytope $P$ is combinatorially equivalent to a product of simplices.
        \nn
  
 $(\textbf{f}) \Rightarrow (\textbf{a}).$
   First of all, \cite[Lemma 2.2]{Uto11} says that there is a homeomorphism
    $\mathcal{Z}_K\cong \R\mathcal{Z}_{L(K)}$.
  Since $K$ has $m$ vertices, $\dim(L(K))= m+\dim(K) = m+n-1$. So if
  $\mathrm{hrk}(\mathcal{Z}_{K};\mathbf{k})=2^{m-n}$, we have 
  $$ \mathrm{hrk}(\R\mathcal{Z}_{L(K)};\mathbf{k}) = 2^{m-n} = 
      2^{2m-(m+n-1)-1} = 2^{2m-\dim(L(K))-1}. $$
  So $(\textbf{e})$ holds for $L(K)$.
   Since we have already shown
    $(\textbf{e}) \Rightarrow (\textbf{a}) \Leftrightarrow (\textbf{b})$,
   $L(K)$ is simplicially isomorphic to 
   $\partial \Delta^{n_1} * \cdots * \partial \Delta^{n_q}$ for some 
               $n_1\cdots, n_q \in \mathbb{N}$. Then we finish the proof by the equivalence
   of $(\textbf{d})$ and $(\textbf{a})$.   
    \end{proof}
    \nn
  
  \begin{rem}
     The equivalences of (\textbf{b}), (\textbf{e}) and (\textbf{f}) 
     in Proposition~\ref{Prop:Polytope-product-2}
      are stated in~\cite[Section 4.8]{BP15} as an exercise.
  \end{rem}
  
 Moreover, we can judge whether a simple polytope $P$ is combinatorially equivalent to 
 a specific product of simplices via some combinatorial invariants called 
 \emph{bigraded Betti numbers}, which are derived from the 
 Stanley-Reisner ring of $P$ (see~\cite[Sec 3.2]{BP15} for the definition).
  Indeed, it is shown in~\cite{SuyPanDong10} that a simple polytope $P$ is 
   combinatorially equivalent to $\Delta^{n_1} \times 
  \cdots \times  \Delta^{n_q}$ if and only if $P$ has the same bigraded Betti numbers as
  $\Delta^{n_1} \times \cdots \times  \Delta^{n_q}$. 
   Simple polytopes with this kind of property are called 
   \emph{combinatorially rigid} (see~\cite{
   SuyKim11, SuyMasDong11}).
\nn
  
   Next, we give a new way to judge whether a simple polytope is combinatorialy equivalent to a product of simplices.\n
   
   \begin{thm} \label{Thm:New-Comb}
    Let $K$ be the boundary of the simplicial polytope dual to a simple polytope $P$.
    Then $P$ is combinatorially equivalent to a product of simplices if and only if
    the following conditions hold for $K$:
   for any maximal simplex $\sigma$ in $K$ and any vertex $v$ of $\sigma$,
               the full subcomplex of $K$ by restricting to $V(K)-V(\sigma)$
               is a simplex of $K$, denoted by $\xi_{\sigma}$, and
                moreover the intersection of $\xi_{\sigma}$ and $\mathrm{link}_Kv$ 
                is a simplex (could be empty) as well.
  \end{thm}     
 \begin{proof}
   Suppose $P$ is a product of simplices. Then $K=
   \partial\Delta^{n_1}*\dots*\partial\Delta^{n_q}$ for some $n_1,\cdots, n_q\in \mathbb{N}$.
 Denote the vertices of $\partial\Delta^{n_k}$ by $v_0^k,v_1^k,\dots,v_{n_k}^k$ for each
 $k=1,\dots,q$. Then for a maximal simplex $\sigma$ in $K$,
  there exists $0\leq l_k \leq n_k$, $k=1,\dots,q$, so that
  $$ \sigma = [v_0^1,\dots,\hat{v}_{l_1}^1,\dots, v_{n_1}^1] * 
  [v_0^2,\dots,\hat{v}_{l_2}^2,\dots, v_{n_2}^2] * \dots * 
  [v_0^q,\dots,\hat{v}_{l_q}^q,\dots, v_{n_q}^q] $$ 
  where $[v_0^k,\dots,\hat{v}_{l_k}^k,\dots, v_{n_k}^k]$ is the simplex spanned by
  all the vertices of $\partial\Delta^{n_k}$ except $v_{l_k}^k$ for each 
  $1\leq k \leq q$. It is easy to see that the full subcomplex of $K$ by 
  restricting to $V(K)-V(\sigma)$ is just the simplex 
  $[v_{l_1}^1, v_{l_2}^2, \dots, v_{l_q}^q]=v_{l_1}^1*v_{l_2}^2\dots *v_{l_q}^q$.
   All the vertices of $\sigma$ are
    $\{ v^{k}_{i_k} \,; \, 0\leq i_k \neq l_k \leq n_k, 1\leq k \leq q\}$. And we have
   $$\link_K v^k_{i_k} = \partial\Delta^{n_1}* \dots *  
    \partial [v_0^k,\dots,\hat{v}_{i_k}^k,\dots, v_{n_k}^k] * \dots * \partial\Delta^{n_q}.  $$
  Note that when $n_k=1$, 
  $\partial [v_0^k,\dots,\hat{v}_{i_k}^k,\dots, v_{n_k}^k]$ is empty.
  Then the intersection of $[v_{l_1}^1, v_{l_2}^2, \dots, v_{l_q}^q]$ and $\link_K v^k_i$ is
  exactly the simplex $[v_{l_1}^1, v_{l_2}^2,\dots, v_{l_q}^q]$ if $n_k>1$, and
  is $[v_{l_1}^1, \dots, \hat{v}_{l_k}^k, \dots, v_{l_q}^q]$ if $n_k=1$. The necessity of 
  these conditions is proved.\n
  
   For the sufficiency, we first show that if these conditions holds for $K$, then they also hold 
   for the link of any vertex of $K$. When $\dim(K)\leq 1$, the theorem is obviously true. So
   we assume $\dim(K)\geq 2$ below. Let $u$ be an arbitrary vertex of $K$. Let $\sigma$ be 
   a maximal simplex of $K$ containing $u$ and let $v$ be an arbitrary vertex of $\sigma$
   different from $u$.
   By our assumption, the intersection
   $\xi_{\sigma}\cap \link_K u$ and $\xi_{\sigma}\cap \link_K v$ are both
   simplices. 
   Let $\tau$ be the simplex with $V(\tau) = V(\sigma)-\{u\}$. Then
   $\tau$ is a maximal simplex in $\link_K u$. 
   Since $V(\xi_{\sigma}) = V(K) - V(\sigma)$, we have
   $$ V(\link_K u) - V(\tau) = V(\xi_{\sigma}) \cap V(\link_K u).$$ Then since
   $\xi_{\sigma}\cap \link_K u$ is a simplex,
   the full subcomplex of $\link_K u$ by restricting to $V(\link_K u) - V(\tau)$
  must agree with $\xi_{\sigma}\cap \link_K u$. Moreover, since $v$ 
   could be any vertex of $\tau$, we need to show that the 
   intersection of the simplex $\xi_{\sigma}\cap \link_K u$
   with $\link_{\link_K u } v$ is also a simplex. 
   Observe that $\link_{\link_K u } v =  \link_K u \cap \link_K v $. So we have
   \begin{align*}
      (\xi_{\sigma}\cap \link_K u) \cap \link_{\link_K u } v &=
      (\xi_{\sigma}\cap \link_K u) \cap (\link_K u \cap \link_K v) \\
       &= (\xi_{\sigma}\cap \link_K u) \cap (\xi_{\sigma} \cap \link_K v)
   \end{align*}
   The intersection of the two simplices $\xi_{\sigma}\cap \link_K u$ and 
   $\xi_{\sigma}\cap \link_K v$ has to be a simplex (could be empty)
    by the definition of simplicial complex.
   Moreover when $\sigma$ ranges over all the maximal simplices of $K$ containing $u$, 
   the vertex $v$ will range
   over all the vertices in $\link_Ku$.
   So our argument shows that these conditions hold for $\link_K u$. \n
   
   By iterating the above argument, we can prove
    that for any codimension-two simplex $\eta$ of $K$,
   the link of $\eta$ in $K$ is a simplicial circle which satisfies the conditions.
   This forces the link of $\eta$ is either $\partial \Delta^2$ or 
   $\partial \Delta^1 * \partial \Delta^1$. Dually it means that any
   $2$-dimensional face of $P$ is either
  a $3$-gon or a $4$-gon. Then by Theorem~\ref{thm:Wiem}(b),
   the polytope $P$ is combinatorially equivalent to a product of simplices.
 \end{proof}

  \nn
    
    Next, we give some descriptions of products of simplices in terms of
  geometric conditions on real moment-angle manifolds of simple convex polytopes. 
     We first recall a concept in metric geometry (see~\cite[Definition 3.1.12]{Burago2001}). \n
      
     \begin{defi}[Quotient Metric Space] \label{Def:Quotient-Metric}
      Let $(X,d)$ be a metric space and let $\mathcal{R}$ be an equivalence relation on $X$.
      The quotient semi-metric $d_\mathcal{R}$ is defined as
      $$d_\mathcal{R}(x,y) = \mathrm{inf}\big\{ \sum^k_{i=1} d(p_i,q_i) \,:\, 
      p_1=x, \, q_k=y,\, k\in \mathbb{N} \big\}, $$
       where the infimum is taken over all choices of $\{p_i\}$ and $\{q_i\}$ such that
      the point $q_i$ is $\mathcal{R}$-equivalent to $p_{i+1}$ for all $i=1,\cdots, k-1$. 
       Moreover, by identifying
      points with zero $d_\mathcal{R}$-distance, we obtain a metric space
       $(X\slash \mathcal{R},d)$ called the \emph{quotient metric space} of $(X,d)$.
    \end{defi} 
        
    Suppose $P$ is a simple convex polytope in a Euclidean space $\R^d$. Consider
    $P$ to be equipped with the \emph{intrinsic metric}. More precisely,
     the intrinsic metric on $P$ defines the distance between any two points $x$ and $y$ in $P$ 
     to be the infimum of lengths of piecewise smooth paths 
     in $P$ that connect $x$ and $y$. Note that the intrinsic metric on 
     $P$ coincides with the subspace metric on $P$ since $P$ is convex.\n
     
     By the construction in~\eqref{Equ:glue-back},
     $\R\mathcal{Z}_P=M(P,\lambda)$ is a closed manifold obtained by gluing $2^m$ 
     copies of $P$ along their facets. We can assume that the $2^m$ copies of $P$ are 
      congruent convex polytopes inside the same Euclidean space
       and the gluings of their facets are all isometries. Then 
      by Definition~\ref{Def:Quotient-Metric} we obtain a quotient metric on $\R\mathcal{Z}_P$,
       denoted by $d_P$.
      It is clear that the metric $d_P$ is 
      invariant with respect to the canonical action of $(\Z_2)^m$ on
     $\R\mathcal{Z}_P$ (see~\eqref{Equ:Canon-Action}).
     
     \begin{rem}
     We can also call $(\R\mathcal{Z}_P,d_P)$  a
      \emph{Euclidean polyhedral space}, 
      which just means that
     it is built from Euclidean polyhedra (see~\cite[Definition 3.2.4]{Burago2001}).
     \end{rem}
     \n
    Note that if $P'$ is another simple convex polytope combinatorially equivalent to $P$ but not
    congruent to $P$, 
     the two metric spaces $(\R\mathcal{Z}_{P'},d_{P'})$ and $(\R\mathcal{Z}_{P},d_P)$ are
     not isometric in general
      (though $\R\mathcal{Z}_{P'}$ is homeomorphic to $\R\mathcal{Z}_{P}$).
\n
  
  \begin{thm} \label{thm:Polytope-product-1}
    Let $P$ be an $n$-dimensional simple convex polytope, 
     $n\geq 2$, with $m$ facets. Then the following statements are all equivalent.
   \begin{itemize}
          \item[(\textbf{a})] $P$ is combinatorially equivalent to a product of simplices.\n
          
          \item[(\textbf{b})] There exists a non-negatively curved Riemannian metric on
          $\R\mathcal{Z}_P$ that is invariant under the canonical $(\Z_2)^m$-action on
          $\R\mathcal{Z}_P$.\n
          
          \item[(\textbf{c})] There exists a simple convex polytope $P'$
         combinatorially equivalent to $P$ 
          so that the metric space $(\R\mathcal{Z}_{P'},d_{P'})$ is non-negatively curved.\n
          
          \item[(\textbf{d})] There exists a simple convex polytope $P'$
         combinatorially equivalent to $P$ so that all the dihedral angles 
         of $P'$ are non-obtuse.
   \end{itemize}   
  \end{thm}   

 Note that a Riemannian metric on a manifold
 is non-negatively curved means that its sectional curvature
 is everywhere non-negative, while a
  metric space being non-negatively curved is defined via comparison
 of triangles (see~\cite[Section 4]{Burago2001}).
  \begin{proof}
    $(\textbf{a}) \Rightarrow (\textbf{b})$\
     The real moment-angle manifold of a product of simplices
   $\Delta^{n_1}\times\cdots \times \Delta^{n_q}$ is diffeomorphic to a product of
   standard spheres $S^{n_1}\times \cdots\times S^{n_q}$ where
    $S^k=\{ (x_1,\cdots, x_{k+1})\in \R^{k+1}\,|\, x^2_1 +\cdots +x^2_{k+1} =1\}$ for any 
    $k\in \mathbb{N}$.
  Let $S^k$ be equipped with the induced Riemannian metric from $\R^{k+1}$. 
   Then it is easy to check that $S^{n_1}\times \cdots\times S^{n_q}$ is a 
  nonnegatively curved Riemannian manifold with respect
  to the product of the Riemannian metrics on $S^{n_1},\cdots, S^{n_q}$.\nn
    
   $(\textbf{b}) \Rightarrow (\textbf{a})$\
    Recall the definition of $\pi_{\lambda}: M(P,\lambda) = \R\mathcal{Z}_P \rightarrow P$
     in~\eqref{Equ:glue-back}. For any proper face $f$ of $P$,
       let $M_f=\pi^{-1}_{\lambda}(f)$. It is easy to see the following.\n
       \begin{itemize}
         \item $M_f$ is an embedded
          closed submanifold of $\R\mathcal{Z}_P$ which has $2^{m+\dim(f)-n-m_f}$
                connected components, where $m_f$ is the number of facets of $f$.\n
         \item Each connected component of $M_f$ is diffeomorphic to $\R\mathcal{Z}_{f}$.   \n 
       \end{itemize}    
      
      Note that $M_f$ is the fixed point set of a rank $n-\dim(f)$ subgroup of $(\Z_2)^m$
      under the canonical action of $(\Z_2)^m$ on $\R\mathcal{Z}_P$. 
      Then since the Riemannian metric is $(\Z_2)^m$-invariant, each component of 
      $M_f$ is a totally geodesic
       submanifold of $\R\mathcal{Z}_P$ (see~\cite[Theorem 5.1]{Kobayashi72}), and so
       is non-negatively curved with respect to 
       the induced Riemannian metric from $\R\mathcal{Z}_P$.
      This implies that the condition (\textbf{b}) holds for $\R\mathcal{Z}_{f}$ as well.\n
      
      In particular when $\dim(f)=2$, the $\R\mathcal{Z}_{f}$ is a closed connected surface with
      non-negatively curved Riemannian metric. Then by Gauss-Bonnet Theorem,
       the Eular characteristic
       $\chi(\R\mathcal{Z}_{f})\geq 0$, which implies that $f$ has to be a 
      $3$-gon or a $4$-gon. Then by Theorem~\ref{thm:Wiem}(\textbf{b}), the polytope 
         $P$ is combinatorially equivalent
         to a product of simplices. \nn
    
   $(\textbf{a}) \Rightarrow (\textbf{c})$ \
   Suppose $P$ is combinatorially equivalent to 
   $\Delta^{n_1}\times\cdots\times\Delta^{n_q}$ where
   $n_1+\cdots+n_q =n$. Consider the standard simplex
   $\Delta^k$ as a metric subspace of $\R^{k+1}$ with the intrinsic metric. 
   Let $P'=\Delta^{n_1}\times\cdots\times\Delta^{n_q}$ be the product of the $q$ metric spaces
       $\Delta^{n_1},\cdots, \Delta^{n_q}$.
   For each $1\leq i \leq q$,
     let $\{ v^i_0,\cdots, v^i_{n_i} \}$ be the set of vertices of 
     $\Delta^{n_i}$. 
    Then all the facets of $P'$ are (see~\cite{SuyMasDong10})
    \[  \{ F^i_{k_i} = \Delta^{n_1}\times \cdots \times \Delta^{n_{i-1}} \times f^{i}_{k_i}
   \times \Delta^{n_{i+1}} \times \cdots \times \Delta^{n_q} \, | \ 0 \leq k_i \leq n_i,
        \ i=1,\cdots, q  \}  
    \]   
   where 
   $f^{i}_{k_i}$ is the codimension-one face of the simplex $\Delta^{n_i}$
   which is opposite to the vertex $v^i_{k_i}$. 
  The total number of facets of $P'$ is $m=n+q$.
    \n
      
    \noindent  \textbf{Claim:} As a metric 
      space $(\R\mathcal{Z}_{P'},d_{P'})$ is isometric to the 
        product of the $q$ metric spaces 
       $(\R\mathcal{Z}_{\Delta^{n_1}},d_{\Delta^{n_1}}),\cdots, 
        (\R\mathcal{Z}_{\Delta^{n_q}},d_{\Delta^{n_q}})$.\n
        
      Indeed if we glue two copies of $P'$ along the facet 
      $F^i_{k_i}$,
     we obtain 
     $$\Delta^{n_1}\times \cdots \times \Delta^{n_{i-1}} \times 
      \big(\Delta^{n_i}\cup_{f^{i}_{k_i}} \Delta^{n_i}\big)
   \times \Delta^{n_{i+1}} \times \cdots \times \Delta^{n_q}.$$
    We can decompose 
   the gluing procedure in the construction~\eqref{Equ:glue-back} for 
   $\R\mathcal{Z}_{P'}$ into $q$ steps. The $i$-th step
   only glues those facets of the form $\{ F^i_{k_i}, 0\leq k_i \leq n_i \}$
    in the $2^m$ copies of $P'$, which 
   gives us the factor
    $(\R\mathcal{Z}_{\Delta^{n_i}},d_{\Delta^{n_i}})$,
    while fixing all other factors in the product.  
  So after the first step we obtain $2^{m-n_1-1}$ copies of
  $\R\mathcal{Z}_{\Delta^{n_1}}\times\Delta^{n_2}\times\cdots\times \Delta^{n_q}$.
   After the second step we obtain $2^{m-n_1-n_2-2}$ copies of 
    $\R\mathcal{Z}_{\Delta^{n_1}}\times \R\mathcal{Z}_{\Delta^{n_2}} 
    \times\Delta^{n_3}\times\cdots\times \Delta^{n_q}$ and so on. Then our claim follows.\n

        Moreover, observe that for any $k\in \mathbb{N}$,
        $(\R\mathcal{Z}_{\Delta^{k}},d_{\Delta^{k}})$ is isometric to
         the boundary of the $(k+1)$-dimensional cross-polytope $Q^{k+1}$ whose vertices are 
         $$\{ (0,\cdots, 0, \overset{i}{1}, 0,\cdots, 0),\,
         (0,\cdots, 0, \overset{i}{-1}, 0,\cdots, 0) \,; \, i=1,\cdots, k+1\}.$$
        Recall that the $n$-dimensional \emph{cross-polytope} is the simplicial polytope dual to
        the $n$-dimensional cube (see Figure~\ref{p:Octahedron} for the case $n=2,3$).\n
        
         \begin{figure}
        \includegraphics[height=2.9cm]{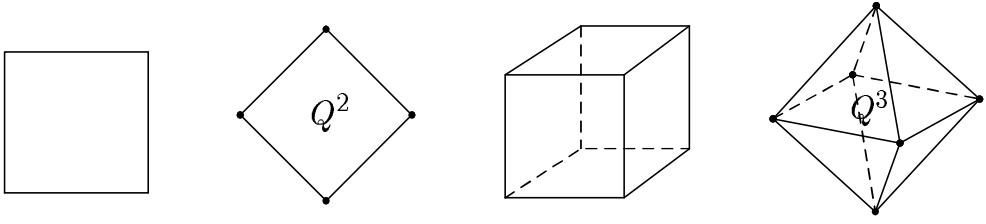}\\
          \caption{Cross-polytopes of dimension $2$ and $3$}\label{p:Octahedron}
      \end{figure}

         It is well known that the intrinsic metric on
         any \emph{convex hypersurface} (i.e. the boundary of a compact 
         convex set with nonempty interior) in a Euclidean space $\R^n$ ($n\geq 3$)
         is non-negatively curved (see~\cite[p.359]{Burago2001}).
         Then since $Q^{k+1}$ is a convex polytope in $\R^{k+1}$, 
          $(\R\mathcal{Z}_{\Delta^{k}},d_{\Delta^{k}})$ is non-negatively curved for any
           $k\geq 2$. When $k=1$, the boundary of $Q^2$ is a 
           piecewise smooth 
           simple curve in $\R^2$. But by definition (see~\cite[Definition 4.1.9]{Burago2001}),
            the intrinsic metric on any piecewise smooth simple curve
           is non-negatively curved because any geodesic triangle on the curve is degenerate. 
           So we can conclude that
          $(\R\mathcal{Z}_{P'},d_{P'})$ is non-negatively curved because 
          the product of non-negatively curved Alexandrov spaces is still
          non-negatively curved (see~\cite[Chapter 10]{Burago2001}).
      \nn

    \noindent $(\textbf{c}) \Rightarrow (\textbf{d})$ \
       If the metric $d_{P'}$ on $\R\mathcal{Z}_{P'}$ 
         is non-negatively curved, we want to show that the dihedral angle between
         any two adjacent facets $F_1$ and $F_2$ of $P'$ is non-obtuse.
         Otherwise, assume that the dihedral angle $\theta$ between $F_1$ and $F_2$ is obtuse.
        Choose a point $O$ in the relative interior of $F_1\cap F_2$, a point $A\in F_1$ and
        $B\in F_2$ so that the line segments $\overline{OA}$ and $\overline{OB}$
         are perpendicular to $F_1\cap F_2$.
        Then $\angle AOB =\theta$. Suppose the lengths of the line segaments $\overline{OA}$,
        $\overline{OB}$ and $\overline{AB}$ are
        $$|\overline{OA}|=|\overline{OB}|=a,\ \ |\overline{AB}|=b.$$ 
            In the gluing construction~\eqref{Equ:glue-back} 
        for $\R\mathcal{Z}_{P'}$, consider two copies of $P'$ glued along the facet $F_1$. We
        then have an isosceles triangle $\triangle AB_1B_2$ in $\R\mathcal{Z}_{P'}$ 
        (see Figure~\ref{p:Comparison}).
        When $a$ is small enough, 
         the distance between $B_1$ and $B_2$ in $(\R\mathcal{Z}_{P'},d_{P'})$ is
         $2a$ by the definition of the quotient metric
        because $\overline{B_1 O}\cup\overline{OB_2}$ is the shortest path 
         between $B_1$ and $B_2$ in $(\R\mathcal{Z}_{P'},d_{P'})$.
         Moreover, let  $\triangle \bar{A}\bar{B}_1\bar{B}_2$ be a
         triangle in the Euclidean plane $\R^2$ which have the same lengths of sides as
         $\triangle AB_1B_2$. Then since $\theta$ is obtuse, it is clear that
         $\triangle AB_1B_2$ is strictly \emph{thinner} than
         $\triangle \bar{A}\bar{B}_1\bar{B}_2$, i.e.
         $$\angle AB_1B_2 < \angle \bar{A}\bar{B}_1\bar{B}_2, \ \
          \angle AB_2B_1 < \angle \bar{A}\bar{B}_2\bar{B}_1,  \ \
          \angle B_1A B_2 < \angle \bar{B}_1\bar{A}\bar{B}_2.$$
          But this contradicts
        our assumption that the metric $d_{P'}$ on $\R\mathcal{Z}_{P'}$ 
         is non-negatively curved (see~\cite[Section 4.1.5]{Burago2001}).
         Therefore, $\theta$ has to be non-obtuse. \nn
         
          \begin{figure}
        \includegraphics[height=4.7cm]{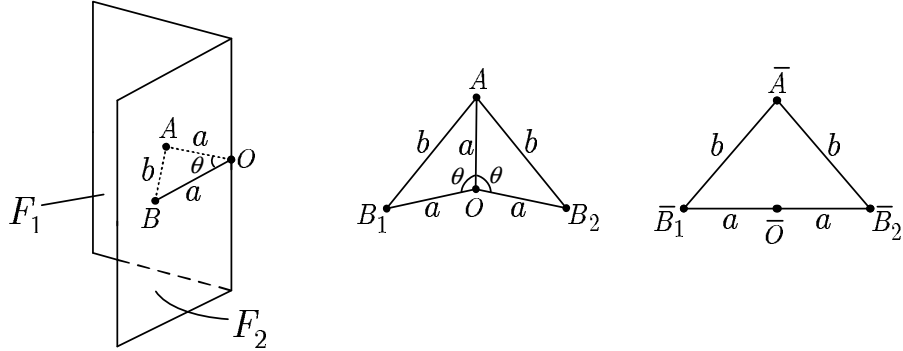}\\
          \caption{Comparison of triangles}\label{p:Comparison}
      \end{figure}

   \noindent $(\textbf{d}) \Rightarrow (\textbf{a})$ \
        Suppose $F_1, F_2$ and $F_3$ are
         three facets of $P'$ with $F_1\cap F_2\cap F_3 \neq
         \varnothing$. Then $F_1\cap F_2$ and $F_1\cap F_3$ are
        codimension-one faces of $F_1$. By our assumption, the dihedral angles of 
        $(F_1,F_2)$, $(F_1,F_3)$ and $(F_2,F_3)$ are all
          non-obtuse. We claim that the dihedral angle between
          $F_1\cap F_2$ and $F_1\cap F_3$ in $F_1$ is non-obtuse as well.\nn

         Indeed, we can assume that $P'$ sits inside $\R^n$
         and let $\eta_i\in \R^n$ ($i=1,2,3$) be a normal vector of
          $F_i$ pointing to the interior of $P$ (see Figure~\ref{p:Cone}).
          By choosing a proper coordinate system of $\R^n$, we can
         assume that $\eta_1 = (0,\cdots, 0 , 1) \in \R^n$ and $F_1$ lies in
         the coordinate hyperplane $\{ x_n=0 \} \subset \R^n$.
         Let $\eta_2=(a_1,\cdots, a_{n-1},a_n)$, $\eta_3=(b_1,\cdots,
         b_{n-1},b_n)$. Since the dihedral angles of $(F_1,F_2)$, $(F_1,F_3)$ and $(F_2,F_3)$ 
         are all non-obtuse, the inner products of 
          $\eta_1,\eta_2, \eta_3$ satisfy
           $$ \eta_1\cdot \eta_2 = a_n\leq 0,\ \ \eta_1\cdot \eta_3 = b_n \leq 0,\ \
         (\eta_2,\eta_3) = a_1b_1 +\cdots + a_{n-1}b_{n-1} + a_n b_n \leq
          0.$$
          \begin{equation} \label{Equ:Inner-Product}
            \Longrightarrow\ \ a_1b_1 +\cdots + a_{n-1}b_{n-1} \leq 0.
          \end{equation}
         Note that $(a_1,\cdots, a_{n-1},0)$ and $(b_1,\cdots, b_{n-1},0)$
         are normal vectors of $F_1\cap F_2$ and $F_1\cap F_3$ inside
         $F_1$ respectively. So~\eqref{Equ:Inner-Product} implies
         that the dihedral angle between $F_1\cap F_2$ and $F_1\cap F_3$ in $F_1$
         is non-obtuse. Our claim is proved.\n
         
       \begin{figure}
        \includegraphics[height=3.7cm]{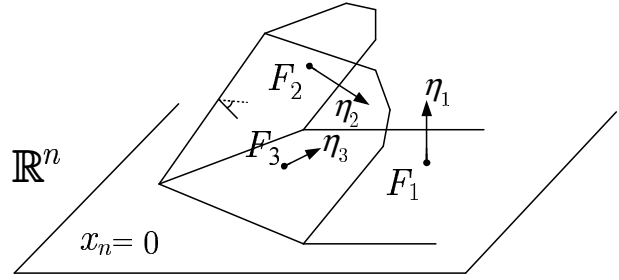}\\
          \caption{Dihedral angles of a simple convex polytope}\label{p:Cone}
      \end{figure}

         By iterating the above arguments, we can show that for any
         $2$-dimensional face $f$ of $P'$,  any interior angle of $f$ is non-obtuse.
         Since $f$ is a Euclidean polygon, it must be either a
         $3$-gon or a $4$-gon. So since $P$ is combinatorially equivalent to $P'$,
         any $2$-face of $P$ is either a $3$-gon or a $4$-gon, too.
         Then by Theorem~\ref{thm:Wiem}(\textbf{b}), the polytope 
         $P$ is combinatorially equivalent
         to a product of simplices.
  \end{proof}
  \n
  
  \begin{rem}
   In the statement of Theorem~\ref{thm:Polytope-product-1}(\textbf{b}),
   if we do not require the 
    Riemannian metric on $\R\mathcal{Z}_P$ to be $(\Z_2)^m$-invariant,
   it is still likely that $P$ has to be combinatorially equivalent to
    a product of simplices (see~\cite[Section 5.2]{KurMasYu15}). 
   But we do not know how to prove this so far. 
  \end{rem}
        \vskip .6cm
    
    \section{Appendix}
   Here we give another proof of Theorem~\ref{thm:Wiem}(\textbf{b}).
    For brevity, we say that a simplicial complex is a \emph{sphere join} if 
    it is isomorphic to $\partial \Delta^{n_1} * \cdots * \partial \Delta^{n_q}$ for some 
   $n_1\cdots, n_q \in \mathbb{N}$.
  One dimensional sphere join is either $\partial \Delta^2$ (boundary of a triangle) or 
  $\partial \Delta^1*\partial \Delta^1$ (boundary of a square). Let us first prove the following
  theorem. 

\begin{thm} \label{theo:1}
  Let $K$ be a simplicial complex of dimension $n$.  Suppose that $K$ satisfies the following 
  two conditions:
\begin{enumerate}
\item[(a)] $K$ is a pseudomanifold,
\item[(b)] the link of any vertex of $K$ is a sphere join of dimension $n-1$,
\end{enumerate} 
Then $K$ is a sphere join. 
\end{thm}

Recall that $K$ is an $n$-dimensional \emph{pseudomanifold} if the following conditions hold:
\begin{itemize}
 \item[(i)] Every $(n-1)$-simplex of $K$ is a face of exactly two $n$-simplices for $n>1$.
 \item[(ii)] For every pair of $n$-simplices $\sigma$ and $\sigma'$ in $K$, 
        there exists
         a sequence of $n$-simplices $\sigma =\sigma_0,\sigma_1,\dots, \sigma_k =\sigma'$
        such that the intersection $\sigma_i\cap \sigma_{i+1}$ is an $(n-1)$-simplex for all $i$.
\end{itemize}
  The Condition (ii) means that $K$ is a \emph{strongly connected} simplicial complex.\nn
  
\begin{proof}
First of all, assumption (b) implies that the link of any $k$-simplex in $K$ is a sphere join of dimension $n-k-1$.
Let $w$ be a vertex of $K$.  By assumption (b), the link $\link_K w$ is of the form
$\link_Kw=\partial\Delta^{n_1}*\dots*\partial\Delta^{n_q}$
where $n_1+\dots+n_q=n$.  Denote the vertices of $\partial\Delta^{n_k}$ by $v_0^k,v_1^k,\dots,v_{n_k}^k$ for $k=1,2,\dots,q$, so that 
\begin{equation} \label{eq:1}
\link_Kw=\partial[v_0^1,v_1^1,\dots, v_{n_1}^1]*\dots*\partial[v_0^q,v_1^q,\dots, v_{n_q}^q].
\end{equation}
Let $I$ be the set of vertices $v_1^1,\dots, v_{n_1}^1,\dots, v_1^q,\dots, v_{n_q}^q$.  Then $[I]$ is a maximal simplex in $\link_Kw$ and the simplex $[I,w]$ spanned by $I$ and $w$ is of dimension $n$.  Since $K$ is a pseudomanifold by assumption (a), there is a unique vertex $v$ in $K$ such that $[I,v]\cap [I,w]=[I]$. We have two cases below. 

\smallskip
Case 1. The case where  $v\notin \link_Kw$. In this case we claim $K=\partial[v,w]*\link_Kw$.  The proof is as follows.  Choose an element from $I$ arbitrarily, say $v^i_j$ $(1\le i\le q$, $1\le j\le n_i)$.  Set $\bar I=(I\backslash\{v^i_j\})\cup\{v^i_0\}$.  Then $[\bar I]$ is an $(n-1)$-simplex of $\link_Kw$ by \eqref{eq:1}, so there is a unique vertex $\bar v$ of $K$ such $[\bar I,\bar v]\cap [\bar I,w]=[\bar I]$ as before since $K$ is a pseudomanifold.  Now we shall observe the link of an $(n-2)$-simplex $[I\cap \bar I]=[I\backslash\{v^i_j\}]$ in $K$.  By our construction, the following are four $n$-simplices in $K$ containing $[I\cap\bar I]$:
\[
[I\cap \bar I,v^i_j,w],\ [I\cap \bar I,v^i_0,w],\ [I\cap \bar I,v^i_j,v],\ 
     [I\cap \bar I,v^i_0,\bar v].
\]
Therefore the vertices $v^i_j,w,v^i_0,v,\bar v$ are in the link of the $(n-2)$-simplex $[I\cap \bar I]$. But by assumption (b), this link is a sphere join of dimension one which can
have at most four vertices. 
Note that $v^i_j,w,v^i_0$ are mutually distinct and $v,\bar v$ are different from $v^i_j,w,v^i_0$. So we must have $\bar v=v$.  
Now let $v^i_j$ run over all elements of $I$, then $\bar I$ runs over all $(n-1)$-simplices in 
$\link_Kw$ that share a $(n-2)$-simplex with $I$. Moreover by~\eqref{eq:1},
 $\link_Kw$ is a strongly connected simplicial complex. So applying
 our argument to $[I]$ and all other $(n-1)$-simplices in $K$, we can 
show that $\partial[v,w]*\link_Kw$ is a subcomplex of $K$.  However, $\partial[v,w]*\link_Kw$ and $K$ are both pseudomanifolds and have the same dimension, so they must agree.  This proves the claim. 
\n

Case 2.  The case where $v\in\link_Kw$, so $v$ is one of $v_0^1,v_0^2,\dots,v_0^q$.  We may assume $v=v_0^1$ without loss of generality.  Then 
\begin{equation} \label{eq:2}
[v,I]=[v_0^1,v_1^1,\dots,v_{n_1}^1,v_1^2,\dots,v_{n_2}^2,\dots,v_1^q,\dots,v_{n_q}^q] \quad\text{is an $n$-simplex in $K$}. 
\end{equation}
We look at $\link_Kv$.  Since $v=v_0^1$, it follows from \eqref{eq:1} that $\link_Kv$ contains 
\begin{equation} \label{eq:3}
\partial[v_1^1,\dots, v_{n_1}^1]*\partial[v_0^2,v_1^2,\dots,v_{n_2}^2]*\dots*\partial[v_0^q,v_1^q,\dots, v_{n_q}^q]
\end{equation}
as a subcomplex.  This together with assumption (b) implies that there is a vertex $w'$ different from vertices in \eqref{eq:3} such that $\link_Kv$ is one of the following:
\[
\begin{split}
&\partial[w',v_1^1,\dots, v_{n_1}^1]*\partial[v_0^2,v_1^2,\dots,v_{n_2}^2]*\dots*\partial[v_0^q,v_1^q,\dots, v_{n_q}^q],\\
&\partial[v_1^1,\dots, v_{n_1}^1]*\partial[w',v_0^2,v_1^2,\dots,v_{n_2}^2]*\dots*\partial[v_0^q,v_1^q,\dots, v_{n_q}^q],\\
&\qquad\qquad\vdots \qquad\qquad \qquad \vdots \\
&\partial[v_1^1,\dots, v_{n_1}^1]*\partial[v_0^2,v_1^2,\dots,v_{n_2}^2]*\dots*\partial[w',v_0^q,v_1^q,\dots, v_{n_q}^q].
\end{split}
\]
However, the fact~\eqref{eq:2} implies that none of the above occurs except the first one.   
So we have
\begin{equation} \label{eq:4}
\link_Kv=\partial[w',v_1^1,\dots, v_{n_1}^1]*\partial[v_0^2,v_1^2,\dots,v_{n_2}^2]*\dots*\partial[v_0^q,v_1^q,\dots, v_{n_q}^q].
\end{equation}

The simplex $[I]$ is in $\link_Kv$ by \eqref{eq:2} and the $n$-simplices $[I,v]$ and $[I,w]$ share $[I]$. Note that $w\in \link_K v$ in this case but $w\neq v^i_j$ for all $i$ and $j$. So 
 from \eqref{eq:4} we can conclude $w=w'$. Then
\begin{equation} \label{eq:5}
\link_Kv=\partial[w,v_1^1,\dots, v_{n_1}^1]*\partial[v_0^2,v_1^2,\dots,v_{n_2}^2]*\dots*\partial[v_0^q,v_1^q,\dots, v_{n_q}^q].
\end{equation}

Remember that $v=v_0^1$.  We claim that $K$ contains 
\begin{equation} \label{eq:6}
\partial[w,v_0^1,v_1^1,\dots, v_{n_1}^1]*\partial[v_0^2,v_1^2,\dots,v_{n_2}^2]*\dots*\partial[v_0^q,v_1^q,\dots, v_{n_q}^q]
\end{equation}
as a subcomplex.  Indeed, any $n$-simplex in \eqref{eq:6} is spanned by $n+1$ vertices which consist of $n_1+1$ vertices from  $\partial[w,v_0^1,v_1^1,\dots, v_{n_1}^1]$ and $n_i$ vertices from $\partial[v_0^i,v_1^i,\dots,v_{n_i}^i]$ for $i=2,3,\dots,q$. Since $v_0^1=v$, either $w$ or $v$ is in the $n_1+1$ vertices from $\partial[w,v_0^1,v_1^1,\dots, v_{n_1}^1]$.  If $w$ (resp. $v$) is in the $n_1+1$ vertices from $\partial[w,v_0^1,v_1^1,\dots, v_{n_1}^1]$, then any $n$-simplex formed this way is in $K$ by \eqref{eq:5} (resp. \eqref{eq:1}).  This proves the claim.  
\n
Finally, since $K$ and the subcomplex \eqref{eq:6} are both pseudomanifolds and have the same dimension, they must agree. So we finish the proof of the theorem.  
\end{proof}

\noindent \textbf{\textit{Proof of Theorem~\ref{thm:Wiem}(\textbf{b})}:}
 Suppose any $2$-dimensional face of $P$ is either a $3$-gon or a $4$-gon. We want to show that
 $P$ is combinatorially equivalent to a product of simplices, or equivalently $\partial P^*$
 is a sphere join.
Let us do induction on the dimension of $P$. When $\dim P=2$, the proof is trivial.
If $\dim P\ge 3$, we will show that $\partial P^*$ satisfies the two conditions in 
Theorem~\ref{theo:1}.  Condition (a) is obvious. By induction assumption, all facets of $P$ are product of simplices which means that $\partial P^*$ satisfies condition (b). 
So we finish the induction by Theorem~\ref{theo:1}. 
\qed
\nn

\begin{rem}
     Theorem~\ref{thm:Wiem}(\textbf{b}) is equivalent to saying that
     if all the facets of $P^n$ are products of simplices, then so is $P^n$.
     In fact this statement has already appeared in an 
     old paper~\cite[Lemma 2.7]{Cox34} where a product of  
     simplices is called a ``simplicial prism''. But the proof of this statement in~\cite{Cox34}
     is a bit vague in the final step. In addition, 
      Theorem~\ref{thm:Polytope-product-1}(\textbf{d})
      is also stated in~\cite[Lemma 2.8]{Cox34}.
  \end{rem}
   \nn

      \section*{Acknowledgement}
      The authors want to thank Hanchul Park and Suyoung Choi for some helpful comments
      and thank Shicheng Xu and Jiaqiang Mei for
    some valuable discussions on the geometry of Alexandrov spaces. 
        \\

\end{document}